\documentclass[a4paper]{amsart}

\usepackage{eucal}
\usepackage[reftex]{theoremref}
\usepackage{tikz-cd}
\numberwithin{equation}{section}
\theoremstyle{plain} \newtheorem{theorem}{Theorem}[section]
\newtheorem{corollary}[theorem]{Corollary}
\newtheorem{lemma}[theorem]{Lemma}
\newtheorem{proposition}[theorem]{Proposition}
\theoremstyle{definition} \newtheorem*{conventions}{Conventions}
\newtheorem{definition}[theorem]{Definition}
\newtheorem{defprop}[theorem]{Definition-Proposition}
\theoremstyle{remark} \newtheorem{remark}[theorem]{Remark}
\newcommand{\ZZ}{\mathbb{Z}}

\newcommand{\FF}{\mathbb{F}} \newcommand{\GG}{\mathbb{G}}

\usepackage{bbm} \newcommand{\id}{\mathbbm{1}}

\newcommand{\kk}{R}
\newcommand{\cf}{\emph{cf.}~}

\newcommand{\lperp}[1]{{}^{\perp_{#1}}}

\newcommand{\set}[1]{\left\{\,#1\,\right\}}
\newcommand{\setP}[2]{\set{#1\,\middle|\,#2}}

\newcommand{\cat}[1]{\mathcal{#1}} \newcommand{\A}{\cat{A}}
\newcommand{\B}{\cat{B}} 
\newcommand{\Q}{\cat{Q}} \newcommand{\M}{\cat{M}}
\newcommand{\U}{\cat{U}} \newcommand{\X}{\cat{X}}
\newcommand{\op}{\mathrm{op}}

\newcommand{\bul}{\bullet}

\newcommand{\smod}{\underline{\mod}\,}
\renewcommand{\mod}{\operatorname{mod}}
\DeclareSymbolFont{sfoperators}{OT1}{cmss}{m}{n}
\DeclareSymbolFontAlphabet{\mathsf}{sfoperators}
\makeatletter\def\operator@font{\mathgroup\symsfoperators}\makeatother

\DeclareMathOperator{\Mod}{Mod}

\DeclareMathOperator{\add}{add} \DeclareMathOperator{\proj}{proj}

\DeclareMathOperator{\coker}{coker}

\DeclareMathOperator{\End}{End} \DeclareMathOperator{\Hom}{Hom}
\DeclareMathOperator{\Ext}{Ext} \DeclareMathOperator{\Tor}{Tor}

\DeclareMathOperator{\Tr}{Tr}

\DeclareMathOperator{\rad}{rad}

\DeclareMathOperator{\gldim}{gl.dim}
\DeclareMathOperator{\domdim}{dom.dim} 

\DeclareMathOperator{\pdim}{proj.dim}
\begin{document}

\date{\today}
  
\title{Higher Auslander correspondence for dualizing $\kk$-varieties}

\author[O.~Iyama]{Osamu Iyama}
\address[Iyama]{Graduate School of Mathematics\\Nagoya University\\
  Chikusa-ku, Furo-cho\\
  Nagoya, 464-8602\\
  JAPAN} \email{iyama@math.nagoya-u.ac.jp}
\urladdr{http://www.math.nagoya-u.ac.jp/~iyama/}

\author[G.~Jasso]{Gustavo Jasso}
\address[Jasso]{Mathematisches Institut\\
  Universit\"at Bonn\\
  Endenicher Allee 60\\
  D-53115 Bonn\\
  GERMANY} \email{gjasso@math.uni-bonn.de}
\urladdr{https://gustavo.jasso.info}

\begin{abstract}
  Let $\kk$ be a commutative artinian ring. We extend higher Auslander
  correspondence from Artin $\kk$-algebras of finite representation
  type to dualizing $\kk$-varieties. More precisely, for a positive
  integer $d$, we show that a dualizing $\kk$-variety is $d$-abelian
  if and only if it is a $d$-Auslander dualizing $\kk$-variety if and
  only if it is equivalent to a $d$-cluster-tilting subcategory of the
  category of finitely presented modules over a dualizing
  $\kk$-variety.
\end{abstract}



\thanks{The authors would like to thank Julian K\"ulshammer for
  suggesting to include a characterization of
  $d\ZZ$-cluster-tilting subcategories and for his comments on a previous
  version of this article. These acknowledgement is extended to
  Chrysostomos Psaroudakis for bringing the article
  \cite{beligiannis_relative_2015} to our attention. The first author
  is supported by JSPS Grant-in-Aid for Scientific Research (B)
  24340004, (C) 23540045 and (S) 15H05738.}

\maketitle

\section{Introduction}

Throughout this article we fix a commutative artinian ring $\kk$.
Recall that an Artin $\kk$-algebra is an $\kk$-algebra which is a
finitely generated $\kk$-module. An Artin $\kk$-algebra $A$ is an
Auslander algebra if
\[
  \gldim A\leq 2\leq \domdim A,
\]
where $\domdim A$ is the dominant dimension of $A$ in the sense of
Tachikawa~\cite{tachikawa_dominant_1964}.  The notion of dominant
dimension was further developed by Auslander's school in their
beautiful theory~\cite{auslander_stable_1969,fossum_trivial_1975,auslander_k-gorenstein_1994,auslander_syzygy_1996} of Auslander--Gorenstein rings.
One of the most important aspects of Auslander--Gorenstein rings was
already given in \cite{auslander_representation_1971}, where Auslander
established a one-to-one correspondence between Auslander algebras and
Artin algebras of finite representation type up to Morita equivalence.

It is natural to extend this correspondence to arbitrary Artin
algebras. This extension is better expressed in the language of
dualizing $\kk$-varieties, which are additive $R$-categories $\A$
enjoying a certain duality between finitely presented $\A$-modules and
finitely presented $\A^\op$-modules. One can think of a dualizing
$\kk$-variety as an analog of the category of finitely generated
projective modules over an Artin algebra, but with possibly infinitely
many indecomposable objects up to isomorphism. Auslander
correspondence can be extended to the following characterization of
categories of finitely presented modules over dualizing
$\kk$-varieties. Although this characterization should be well known
to specialists, we did not find it in the literature. We refer the
reader to Section \ref{sec:preliminaries} for the definitions of
dualizing $\kk$-variety and Auslander dualizing $\kk$-variety.

\begin{theorem}
  \th\label{main_thm_d1} Let $\A$ be a dualizing $\kk$-variety. Then,
  the following statements are equivalent.
  \begin{enumerate}
  \item The category $\A$ is abelian.
  \item There exist a dualizing $\kk$-variety $\B$ and an equivalence
    $\A\cong\mod\B$.
  \item The category $\A$ is an Auslander dualizing $\kk$-variety.
  \end{enumerate}
\end{theorem}

In fact, \th\ref{main_thm_d1} is a particular case of a more general
result, \th\ref{auslander_correspondence} below, which gives a
characterization of $d$-cluster-tilting subcategories of categories of
finitely presented modules over dualizing $\kk$-varieties. This result
is motivated by a ``higher dimensional'' version of Auslander
correspondence which we briefly recall.

From the viewpoint of higher dimensional Auslander--Reiten theory, the
first author introduced in \cite{iyama_auslander_2007} the class of
$d$-Auslander algebras, which are the Artin $\kk$-algebras $A$ such that
\[
  \gldim A\leq d+1\leq \domdim A,
\]
where $d$ is a positive integer. Thus, if $d=1$, then one recovers the
classical Auslander algebras. Moreover, a one-to-one correspondence
was established in \cite[Thm. 02]{iyama_auslander_2007} between Morita
equivalence classes of $d$-Auslander algebras and equivalence classes
of $d$-cluster-tilting subcategories with additive generators of
categories of finitely presented modules over Artin algebras. One can
think of a $d$-cluster-tilting subcategory as a higher analog of the
module category. The notion of $d$-abelian category was introduced in
\cite{jasso_n-abelian_2014} in order to make precise part of this
analogy. Indeed, $d$-cluster-tilting subcategories are typical
examples of $d$-abelian categories.

The following natural generalization of
\cite[Thm. 02]{iyama_auslander_2007} to the setting of dualizing
$\kk$-varieties is one of the main results of this article. We refer
the reader to Section \ref{sec:preliminaries} for the definitions of
$d$-abelian category, $d$-Auslander dualizing $\kk$-variety, and
$d$-cluster-tilting subcategory.

\begin{theorem}[Auslander correspondence]
  \th\label{auslander_correspondence} Let $\A$ be a dualizing
  $\kk$-variety and $d$ a positive integer. Then, the following
  statements are equivalent.
  \begin{enumerate}
  \item\label{it:d-abelian} The category $\A$ is $d$-abelian.
  \item\label{it:d-ct} There exist a dualizing $\kk$-variety $\B$ and
    a fully faithful functor $\FF\colon\A\to\mod\B$ such that $\FF\A$
    is a $d$-cluster-tilting subcategory of $\mod\B$.
  \item\label{it:d-auslander} The category $\A$ is a $d$-Auslander
    dualizing $\kk$-variety.
  \item\label{it:new-condition} We have $\gldim\A\le d+1$,
    ${}^\perp \A_\A\subset {}^{\perp_d}\A_\A $ and
    $ {}^\perp \A_{A^{\op}}\subset {}^{\perp_d}\A_{\A^{\op}}$.
  \end{enumerate}
\end{theorem}

Note that the implication
\eqref{it:d-ct}$\Rightarrow$\eqref{it:d-abelian} in
\th\ref{auslander_correspondence} is shown in
\cite[Thm. 3.16]{jasso_n-abelian_2014}. Also, we mention that a
partial version of \th\ref{auslander_correspondence}, which does not
involve the relationship with $d$-abelian categories (one of the main
points of our result), is shown in
\cite[Thm. 8.4]{beligiannis_relative_2015} in a more general setting
than that of dualizing $\kk$-varieties.

As a consequence of \th\ref{auslander_correspondence}, we obtain a
characterization of $d\ZZ$-cluster-tilting subcategories,
that is those $d$-cluster-tilting subcategories which are closed under
$d$-syzygies and $d$-cosyzygies.  We refer the reader to Section
\ref{sec:preliminaries} for definitions of
$d\ZZ$-cluster-tilting subcategory and of $d$-abelian categories having
$d$-syzygies (resp. $d$-cosyzygies).

\begin{theorem}[Homological Auslander correspondence]
  \th\label{auslander_correspondence_h} Let $\A$ be a dualizing
  $\kk$-variety and $d$ a positive integer.  Then, the following
  statements are equivalent.
  \begin{enumerate}
  \item\label{it:d-abelian_h} The category $\A$ is $d$-abelian and has
    $d$-cosyzygies.
  \item\label{it:d-ct_h} There exist a dualizing $\kk$-variety $\B$
    and a fully faithful functor $\FF\colon\A\to\mod\B$ such that
    $\FF\A$ is a $d\ZZ$-cluster-tilting subcategory of
    $\mod\B$.
  \item\label{it:d-auslander_h} The category $\A$ is a $d$-Auslander
    dualizing $\kk$-variety satisfying 
    \[
      \Omega^{-d}(\A_\A)\subseteq \Omega(\lperp{}\A_\A).
    \]
  \end{enumerate}
  Moreover, the three equivalent statements above are equivalent to
  the following statements:
  \begin{enumerate}
  \item[(a$^\op$)] The category $\A$ is $d$-abelian and has
    $d$-syzygies.
  \item[(b$^\op$)] There exist a dualizing $\kk$-variety $\B$
    and a fully faithful functor $\FF\colon\A^\op\to\mod\B$ such that
    $\FF\A^\op$ is a $d\ZZ$-cluster-tilting subcategory of
    $\mod\B$.
  \item[(c$^\op$)] The category $\A$ is a $d$-Auslander dualizing
    $\kk$-variety satisfying 
    \[
      \Omega^d(D\A_\A)\subseteq \Omega^{-1}((D\A_\A)^\perp).
    \]
  \end{enumerate}
\end{theorem}

Note that the implication
\eqref{it:d-ct_h}$\Rightarrow$\eqref{it:d-abelian_h} in
\th\ref{auslander_correspondence_h} is shown in
\cite[Thm. 5.16]{jasso_n-abelian_2014}.

Conceptually,
\th\ref{auslander_correspondence,auslander_correspondence_h} enhance
our understanding of the relationship between $d$-abelian categories
and $d$-cluster-tilting subcategories.

Finally, let us describe the structure of the article. In Section
\ref{sec:preliminaries} we recall the definitions and results needed
in the sequel. In Section \ref{sec:auslander_correspondence} we give a
proof of \th\ref{auslander_correspondence}. In Section
\ref{sec:auslander_correspondence_h} we give a proof
\th\ref{auslander_correspondence_h}. Finally, in Section
\ref{sec:examples} we provide examples illustrating our results.

\begin{conventions}
  We fix a commutative artinian ring $\kk$ together with a positive
  integer $d$. We denote by $\mod\kk$ the category of finitely
  presented $\kk$-modules. Let $I$ be the injective envelope of
  $\kk/\rad\kk$ and $D:=\Hom_\kk(-,I)$ the usual duality. All
  categories we consider are assumed to be additive and
  Krull--Schmidt, that is every object decomposes as a finite direct
  sum of objects whose endomorphism rings are local.  Moreover, all
  categories are assumed to be $\kk$-linear and all functors are
  assumed to be $\kk$-linear and additive. Let $\A$ be an
  $\kk$-category and $a,a'\in\A$. We denote the $\kk$-module of
  morphisms $a\to a'$ by $\A(a,a')$. We denote the Yoneda embedding by
  $a\mapsto P_a:=\A(-,a)$. By subcategory we mean full subcategory
  which is closed under isomorphisms.  Under our standing assumptions,
  for an object $a\in\A$ we denote its additive closure by $\add a$.
  Recall that $\add a$ is the smallest subcategory of $\A$ containing
  $a$ and which is closed under finite direct sums and direct
  summands.
\end{conventions}

\section{Preliminaries}
\label{sec:preliminaries}

In this section we recall the notion of dualizing $\kk$-variety
introduced by Auslander and Reiten in \cite{auslander_stable_1974}. We
also recall the definitions of $d$-cluster-tilting subcategory and
$d$-abelian category as well as technical results which are needed in
the proofs of the main theorems.

\subsection{Dualizing $\kk$-varieties}

We begin by recalling the basics on functor categories and finitely
presented modules. We refer the reader to
\cite{auslander_coherent_1966} for a thorough development of these
concepts.

Let $\A$ be an essentially small category. We assume that $\A$ is a
$\Hom$-finite $\kk$-category, that is for all $a,a'\in\A$ the
$\kk$-module $\A(a,a')$ is finitely generated and the composition
\[
  \begin{tikzcd}
	\A(a,a')\otimes\A(a',a'')\rar&\A(a,a'')
  \end{tikzcd}
\]
is $\kk$-bilinear.  A \emph{(right) $\A$-module} is a contravariant
$\kk$-linear additive functor $\A\to\Mod\kk$; a \emph{morphism}
$M\to N$ between $\A$-modules $M$ and $N$ is a natural
transformation. Thus, we obtain an abelian category of $\A$-modules
denoted by $\Mod\A$. If $M,N$ are $\A$-modules, we denote the set
$(\Mod\A)(M,N)$ of natural transformations $M\to N$ by $\Hom_\A(M,N)$,
which is an $\kk$-module in a natural way.

An $\A$-module $M$ is \emph{finitely generated} if there exist an
epimorphism $P_a\to M$ for some $a\in\A$.  Via the Yoneda embedding we
identify $\A$ with the full subcategory $\A_\A$ of $\Mod\A$ of
\emph{finitely generated projective $\A$-modules}. With some abuse of
notation, we write $\A_{\A^\op}:=\A^\op_{\A^\op}$ for the category of
finitely generated projective $\A^\op$-modules, which is a subcategory
of $\Mod(\A^\op)$.  An $\A$-module $M$ is \emph{finitely presented} if
there exist a morphism $f\colon a\to a'$ in $\A$ and an exact sequence
\[
  \begin{tikzcd}[column sep=small]
	P_a\rar{P_f}&P_{a'}\rar&M\rar&0.
  \end{tikzcd}
\]
We denote the full subcategory of $\Mod\A$ of finitely presented
$\A$-modules by $\mod\A$. Note that $\mod\A$ is closed under cokernels
and extensions (by the Horseshoe Lemma) in $\Mod\A$, and that it is
closed under kernels in $\Mod\A$ if and only if $\A$ has weak kernels,
see \cite{auslander_coherent_1966}. Therefore, $\mod\A$ is an exact
abelian subcategory of $\Mod\A$ if and only if $\A$ has weak kernels.

\begin{definition}
  \cite{auslander_stable_1974} An essentially small $\kk$-linear
  additive Krull--Schmidt category $\A$ is a \emph{dualizing
    $\kk$-variety} if the contravariant functor
  $\Mod\A\to\Mod(\A^\op)$ given by $M\mapsto D\circ M$ induces a
  duality $D\colon\mod\A\to\mod(\A^\op)$.
\end{definition}

Let $\A$ be a dualizing $\kk$-variety. Since $\mod\A$ and
$\mod(\A^\op)$ have cokernels, it follows from the existence of a
duality $\mod\A\to\mod(\A^\op)$ that $\mod\A$ and $\mod(\A^\op)$ are
closed under kernels in $\Mod\A$ and $\Mod(\A^\op)$
respectively. Therefore $\mod\A$ and $\mod(\A^\op)$ are abelian
categories with enough projectives and injectives, see
\cite[Thm. 2.4]{auslander_stable_1974}.

The most basic examples of dualizing $\kk$-varieties arise from Artin
algebras. Recall that if $A$ is an Artin algebra, then there is an
equivalence $\mod A\cong\mod(\proj A)$.

\begin{proposition}
  \cite[Prop. 2.5]{auslander_stable_1974} Let $A$ be an Artin
  algebra. Then, the category $\proj A$ of finitely generated
  projective $A$-modules is a dualizing $\kk$-variety.
\end{proposition}

The importance of the concept of dualizing $\kk$-variety comes from
the following result which was instrumental in the first proof of the
existence theorem of almost-split sequences. It allows us to
investigate dualizing $\kk$-varieties and their categories of finitely
presented modules using the same representation-theoretic methods.

\begin{proposition}
  \cite[Prop. 2.6]{auslander_stable_1974} Let $\A$ be a dualizing
  $\kk$-variety.  Then, $\mod\A$ is a dualizing $\kk$-variety.
\end{proposition}

Let $\A$ be a dualizing $\kk$-variety. Recall that a subcategory $\B$
of $\A$ is \emph{contravariantly finite} if for all $a\in\A$ there
exist $b\in\B$ and a morphism $f\colon b\to a$ such that the sequence
\[
  \begin{tikzcd}[column sep=small]
	\A(-,b)|_\B\rar{P_f}&\A(-,a)|_\B\rar&0
  \end{tikzcd}
\]
is exact. Such a morphism $f$ is called a \emph{right
  $\B$-approximation of $a$}. \emph{Covariantly finite} subcategories
of $\A$ are defined dually. The subcategory $\B$ is \emph{functorially
  finite} if it is both contravariantly finite and covariantly
finite. The following well known result is a basic tool for
constructing dualizing $\kk$-varieties, see
\cite[Thm. 2.3]{auslander_almost_1981} and
\cite[Prop. 1.2]{iyama_auslander_2007} for a general statement.

\begin{proposition}
  \th\label{dkv_ff} Let $\A$ be a dualizing $\kk$-variety and $\B$ a
  functorially finite subcategory of $\A$. Then, $\B$ is a dualizing
  $\kk$-variety.
\end{proposition}

\begin{lemma}
  \th\label{lemma_DA} Let $\A$ be a dualizing $\kk$-variety. Then,
  $\A_\A$ and $(D\A)_\A$ are functorially finite in $\mod\A$.
\end{lemma}
\begin{proof}
  By duality, it suffices to show that $\A_\A$ is functorially finite
  in $\mod\A$.  Contravariantly finiteness is clear. Fix $M$ in
  $\mod\A$. Since $M^*$ is in $\mod(\A^{op})$, we can take a
  surjection $P_a^*\to M^*$. It is easy to check that the composition
  $M\to M^{**}\to P_a$ is a left $\A_\A$-approximation of $M$.
\end{proof}

We recall the following fundamental property of dualizing
$\kk$-varieties.

\begin{proposition}
  \th\label{minimal_proj_res} \cite[Prop. 3.4]{auslander_stable_1974}
  Let $\A$ be a dualizing $\kk$-variety. Then every finitely presented
  $\A$-module has a minimal projective (resp. injective) presentation
  (resp. copresentation). In particular, $\mod\A$ has projective
  covers and injective envelopes.
\end{proposition}

Let $\A$ be a dualizing $\kk$-variety. We recall the construction of
the Auslander--Bridger transpose of an $\A$-module. Firstly, Yoneda's
lemma implies that the contravariant left exact functor
\[
  \begin{tikzcd}
    (-)^*\colon\Mod\A\rar&\Mod(\A^\op)
  \end{tikzcd}
\]
defined by
\[
  M^*:=\Hom_\A(M,-)|_{\A}
\]
induces a duality $(-)^*\colon\A_\A\to\A_{\A^\op}$ which satisfies
$P_a^*\cong\A(a,-)$. We call this duality the
\emph{$\A$-duality}. Secondly, let $M\in\mod\A$ and choose a
projective presentation
\[
  \begin{tikzcd}[column sep=small]
	P_{a_1}\rar{P_f}&P_{a_0}\rar&M\rar&0
  \end{tikzcd}
\]
and set $\Tr M:=\coker(P_f^*)$. Finally, in order to extend $\Tr$ to a
functor, denote by $\smod\A$ the quotient of the category $\mod\A$ by
the ideal of morphisms which factor through a finitely generated
projective $\A$-module.  Using the lifting property of projective
$\A$-modules it is easy to see that this association induces a well
defined functor $\Tr\colon\smod\A\to\smod(\A^\op)$ which is called the
\emph{Auslander--Bridger transposition}.

The following result is well known in the case of Artin algebras, see
\cite[Prop. 6.3]{auslander_coherent_1966}. We need it in the more
general setting of dualizing $\kk$-varieties.  Recall that Heller's
syzygy functor $\Omega\colon\smod\A\to\smod\A$ is defined by a short
exact sequence
\[
  \begin{tikzcd}[column sep=small]
	0\rar&\Omega M\rar&P_{a_0}\rar&M\rar&0.
  \end{tikzcd}
\]
The cosyzygy functor
$\Omega^{-1}\colon\overline{\mod}\,\A\to\overline{\mod}\,\A$ is
defined dually.

\begin{proposition}[Auslander--Bridger sequence]
  \th\label{auslander-bridger_sequence} Let $\A$ be a dualizing
  $\kk$-variety. For each $M\in\mod\A$ there exists an exact sequence
  \[
    \begin{tikzcd}[column sep=small]
      0\rar&\Ext_{\A^\op}^1(\Tr
      M,-)|_{\A^\op}\rar&M\rar&M^{**}\rar&\Ext_{\A^\op}^2(\Tr
      M,-)|_{\A^\op}\rar&0
    \end{tikzcd}
  \]
  Moreover, $M^{**}\in\Omega^2(\mod\A)$.
\end{proposition}
\begin{proof}
  The proof of \cite[Prop. 6.3]{auslander_coherent_1966} carries
  over. We give a direct proof for the convenience of the reader.  Let
  $P_a\to P_{a'}\to M\to0$ be a projective presentation of $M$. By
  definition, the $\A$-duality yields is an exact sequence
  \[
    \begin{tikzcd}[column sep=small]
      0\rar&M^*\rar&P_{a'}^*\rar&P_a^*\rar&\Tr M\rar&0.
    \end{tikzcd}
  \]
  Let $P_{b}^*\to P_{b'}^*\to M^*\to0$ be a projective presentation of
  $M^*$. Thus, we obtain a commutative diagram
  \[
    \begin{tikzcd}[row sep=small]
      P_a^{**}\rar\dar&P_{a'}^{**}\rar\dar&P_{b'}^{**}\rar&P_{b}^{**}\\
      0\rar&M\dar\rar&M^{**}\rar\uar&0\uar\\
      &0&0\uar
    \end{tikzcd}
  \]
  in which the sequences $P_a^*{**}\to P_{a'}^{**}\to M\to 0$ and
  $0\to M^{**}\to P_{b'}^{**}\to P_b^{**}$ are exact. Moreover, it is
  readily seen that the kernel of $M\to M^{**}$ is isomorphic to the
  cohomology of the top row at $P_{a'}^{**}$ which is isomorphic to
  $\Ext_{\A^\op}^1(\Tr M,-)|_{\A^\op}$; similarly, the cokernel of
  $M\to M^{**}$ is isomorphic to the cohomology of the top row at
  $P_{b'}^{**}$ which is isomorphic
  $\Ext_{\A^\op}^2(\Tr M,-)|_{\A^\op}$. This yields the required exact
  sequence.  The second claim follows immediately from the
  construction of the Auslander--Bridger transposition.
\end{proof}

For each $k\geq1$ we consider the functor
$\Tr_k:=\Tr\Omega^{k-1}\colon\smod\A\to\smod(\A^\op)$. These functors
are instrumental in higher Auslander--Reiten theory, see
\cite{iyama_higher-dimensional_2007}. We need the following well-known
property.

\begin{proposition}
  \th\label{technical_lemma} Let $\A$ be a dualizing
  $\kk$-variety. Then, for each $M\in\mod\A$ and for each $k\ge1$
  there is an exact sequence
  \[
    \begin{tikzcd}[column sep=small]
      0\rar&\Ext_\A^k(M,-)|_{\A}\rar{\varphi}&\Tr_k
      M\rar&\Omega\Tr_{k+1}M\rar&0.
    \end{tikzcd}
  \]
  such that $\varphi^*=0$.
\end{proposition}
\begin{proof}
  We include a proof for the convenience of the reader. The first part
  of the proof is analogous to the proof of
  \th\ref{auslander-bridger_sequence}. Let $M\in\mod\A$ and
  $P_\bul\to M$ a projective resolution of $M$. For each $k\geq1$ the
  $\A$-duality yields an exact sequence
  \[
    \begin{tikzcd}[column sep=small]
      0\rar&(\Omega^k M)^*\rar&P_k^*\rar&P_{k+1}^*\rar&\Tr_{k+1}
      M\rar&0
    \end{tikzcd}
  \]
  It is readily verified that there exist a commutative diagram with
  exact rows and columns
  \[
    \begin{tikzcd}[column sep=small]
      P_{k-1}^*\dar[equals]\rar&(\Omega^kM)^*\dar[tail]\rar&\Ext_\A^1(\Omega^{k-1}M,-)|_\A\dar[dotted,tail]{\varphi}\rar&\Ext_\A^1(P_{k-1},-)|_\A=0\\
      P_{k-1}^*\dar\rar&P_k^*\dar[two heads]\rar[two heads]&\Tr_kM\dar[dotted,two heads]{\psi}\\
      0\rar&\Omega\Tr_{k+1}M\rar[equals]&\Omega\Tr_{k+1}M
    \end{tikzcd}
  \]
  where the dotted column can be seen to be exact by applying the
  Snake Lemma to the leftmost two columns. The first claim follows
  since $\Ext_\A^1(\Omega^{k-1}M,-)|_{\A}\cong\Ext_\A^k(M,-)|_\A$.
  
  It remains to show that $\varphi^*=0$. Thus, we need to show that
  every morphism $\Tr_k M\to P^*$ where $P$ is a projective
  $\A$-module factors through $\psi$. Equivalently, we need to show
  that every morphism $P_k^*\to P^*$ such that the composition with
  $P_{k-1}^*\to P_k^*$ vanishes factors through
  $P_k^*\to \Omega\Tr_{k+1} M$. Indeed, let $f^*\colon P_k^*\to P^*$
  be such a morphism. Then, since $P$ is projective and the complex
  $P_\bul\to M$ is exact, there is a commutative diagram
  \[
    \begin{tikzcd}[column sep=small]
      &P\dar{f}\drar{0}\dlar[dotted,swap]{g}\\
      P_{k+1}\rar&P_k\rar&P_{k-1}
    \end{tikzcd}
  \]
  By applying the $\A$-duality to this diagram we deduce that $f^*$
  factors through $P_k^*\to P_{k+1}^*$, which implies the required
  factorization.  This shows that $\varphi^*=0$.
\end{proof}

Let $\A$ be a dualizing $\kk$-variety. We recall from
\cite{auslander_stable_1974} that there is a unique bifunctor
$-\otimes_\A-\colon \Mod\A\times\Mod(\A^\op)\to\Mod\kk$, called of
course the \emph{tensor product}, characterized by the following
properties:
\begin{enumerate}
\item Let $M\in\Mod\A$. The functor
  $M\otimes_\A-\colon\Mod(\A^\op)\to\Mod\kk$ is right exact, commutes
  with direct sums and for each $a\in\A$ there is an equality
  $M\otimes_\A P_a^*=\Hom_\A(P_a,M)$.
\item Let $N\in\Mod(\A^\op)$. The functor
  $-\otimes_\A N\colon\Mod\A\to\Mod\kk$ is right exact, commutes with
  direct sums and for each $a\in\A$ there is an equality
  $P_a\otimes_\A N=\Hom_{\A^\op}(P_a^*,N)$.
\end{enumerate}
For an arbitrary $\A^\op$-module $N$, the functors
$\Tor_k^\A(-,N)\colon\mod\A\to\Mod\kk$ are defined as usual, that is
as the left derived functors of $-\otimes_\A N$. We need the following
well known isomorphism from homological algebra,
\cf~\cite[Prop. 5.3]{cartan_homological_1999}.

\begin{lemma}
  \th\label{tor_homext} Let $\A$ be a dualizing $\kk$-variety,
  $M\in\mod\A$ and $I$ an injective $\A^\op$-module. Then, for each
  $k\geq 0$ there is a natural isomorphism
  \[
    \Tor_k^\A(M,I)\cong\Hom_{\A^\op}(\Ext_\A^k(M,-)|_\A,I).
  \]
\end{lemma}
\begin{proof}
  We give a proof for the convenience of the reader.  Let
  $P_\bul\to M$ be a projective resolution of $M$. On one hand, for
  each $k\geq0$ the homology of the complex
  \begin{equation}
    \label{eq:tor}
    \begin{tikzcd}[column sep=small]
      \cdots\rar&P_k\otimes_\A I\rar&\cdots\rar&P_1\otimes_\A
      I\rar&P_0\otimes_\A I\rar&0
    \end{tikzcd}
  \end{equation}
  at $P_k\otimes_\A I$ is isomorphic to $\Tor_k^\A(M,I)$.

  On the other hand, for each $k\geq0$ the homology of the complex
  \[
    \begin{tikzcd}[column sep=small]
      0\rar&P_0^*\rar&P_1^*\rar&\cdots\rar&P_k^*\rar&\cdots
    \end{tikzcd}
  \]
  at $P_k^*$ is isomorphic to $\Ext_\A^k(M,-)|_\A$. Since $I$ is
  injective, the contravariant functor $\Hom_{\A^\op}(-,I)$ is exact,
  hence it preserves homology. Thus, for each $k\geq0$ the homology of
  the complex
  \begin{equation}
    \label{eq:hom_ext}
    \begin{tikzcd}[column sep=tiny]
      \cdots\rar&\Hom_{\A^\op}(P_k^*,I)\rar&\cdots\rar&\Hom_{\A^\op}(P_1^*,I)\rar&\Hom_{\A^\op}(P_0^*,I)\rar&0.
    \end{tikzcd}
  \end{equation}
  at $\Hom_{\A^\op}(P_k^*,I)$ is isomorphic to
  $\Hom_{\A^\op}(\Ext_\A^k(M,-)|_\A,I)$.

  Finally, by the definition of the tensor product, the complexes
  \eqref{eq:tor} and \eqref{eq:hom_ext} are isomorphic, therefore they
  have isomorphic homologies. This finishes the proof.
\end{proof}

Let $\A$ be a dualizing $\kk$-variety. Recall that the \emph{global
  dimension of $\A$}, denoted by $\gldim\A$, is the supremum of all
the projective dimensions of finitely presented $\A$-modules. The
duality $D\colon\mod\A\to\mod(\A^\op)$ implies
$\gldim\A=\gldim\A^\op$. We also recall the definition of the
\emph{dominant dimension} of a dualizing $\kk$-variety. Let $\A$ be a
dualizing $\kk$-variety and $d$ a positive integer. We say that
$\domdim\A\geq d+1$ if for all $a\in\A$ there exists an injective
coresolution
\[
  \begin{tikzcd}[column sep=small]
	0\rar&P_a\rar&I^0\rar&I^1\rar&\cdots\rar&I^d\rar&\cdots
  \end{tikzcd}
\]
such that $I^0,I^1,\dots,I^d$ are projective $\A$-modules.  As in the
case of artin algebras \cite{tachikawa_dominant_1964}, there is an equality
\[
  \domdim\A=\domdim\A^{\op}.
\]
This is a consequence of the left-right symmetry of
Auslander's $k$-Gorenstein property \cite[Thm.
3.7(c)$\Leftrightarrow$(d)]{fossum_trivial_1975} and a categorical version of
\cite[Thm. 1.1]{iyama_symmetry_2003}.

\begin{definition}
  Let $\A$ be a dualizing $\kk$-variety and $d$ a positive integer. We
  say that $\A$ is a \emph{$d$-Auslander dualizing $\kk$-variety} if
  \[
    \gldim\A\leq d+1\leq\domdim\A.
  \]
  If $d=1$, then we simply say that $\A$ is an \emph{Auslander
    dualizing $\kk$-variety}.
\end{definition}

\subsection{$d$-cluster-tilting subcategories}

We now recall the definition of $d$-cluster-tilting subcategory. For
convenience, we introduce the following notation.  Let $\A$ be a
dualizing $\kk$-variety. Given a subcategory $\X$ of $\mod\A$, we
define the subcategories
\[
  \lperp{}\X:=\setP{M\in\mod\A}{\Hom_\A(M,\X)=0}
\]
and, for $d\geq1$,
\[
  \lperp{d-1}\X:=\setP{M\in\mod\A}{\forall
    k\in\set{1,\dots,d-1}\,\Ext_\A^k(M,\X)=0}.
\]
The subcategories $\X^\perp$ and $\X^{\perp_{d-1}}$ are defined
dually. Note that $\lperp{0}{\X}=\mod\A$, hence $\lperp{}\X$ and
$\lperp{0}\X$ are different in general. The subcategory $\X$ is
\emph{$d$-rigid} if $\X\subseteq\lperp{d-1}\X$.

\begin{definition}
  \cite[Def. 2.2]{iyama_higher-dimensional_2007} Let $\A$ be a
  dualizing $\kk$-variety, $\M\subseteq\mod\A$ a functorially finite
  subcategory and $d\geq1$. We say that $\M$ is
  \emph{$d$-cluster-tilting} if the equalities
  $\M=\lperp{d-1}\M=\M^{\perp_{d-1}}$ hold.
\end{definition}

\begin{remark}
  Let $\A$ be a dualizing $\kk$-variety. Then, $\mod\A$ has a unique
  $1$-cluster-tilting subcategory, namely $\mod\A$ itself.
\end{remark}

We need the following characterization of $d$-cluster-tilting
subcategories.

\begin{proposition}
  \th\label{char_tilting}
  \cite[Prop. 2.2.2]{iyama_higher-dimensional_2007} Let $\A$ be a
  dualizing $\kk$-variety and $\M\subseteq\mod\A$ a functorially
  finite subcategory. Then, the following statements are equivalent.
  \begin{enumerate}
  \item The subcategory $\M$ is $d$-cluster-tilting.
  \item There is an equality $\M=\lperp{d-1}\M$ and $\M$ contains all
    injective $\A$-modules.
  \item There is an equality $\M=\M^{\perp_{d-1}}$ and $\M$ contains
    all projective $\A$-modules.
  \end{enumerate}
\end{proposition}

We recall the following property of $d$-cluster-tilting subcategories,
which exposes their higher homological nature.

\begin{proposition}
  \th\label{not_long} \cite[Lemma 3.5]{iyama_cluster_2011} Let $\A$ be
  a dualizing $\kk$-variety and $\M\subseteq\mod\A$ a
  $d$-cluster-tilting subcategory. Then, for each $X\in\M$ and for
  each
  \[
    \begin{tikzcd}[column sep=small]
      0\rar&L\rar&M^1\rar&\cdots\rar&M^d\rar&N\rar&0
    \end{tikzcd}
  \]
  exact sequence in $\mod\A$ whose terms lie in $\M$ there are exact
  sequences
  \[
    \begin{tikzcd}[column sep=tiny, row sep=tiny]
      0\rar&\Hom_\A(X,L)\rar&\Hom_\A(X,M^1)\rar&\cdots\rar&\Hom_\A(X,M^d)\rar&\Hom_\A(X,N)\rar&{}\\
      {}\rar&\Ext_\A^d(X,L)\rar&\Ext_\A^d(X,M^1)
    \end{tikzcd}
  \]
  and
  \[
    \begin{tikzcd}[column sep=tiny, row sep=tiny]
      0\rar&\Hom_\A(N,X)\rar&\Hom_\A(M^d,X)\rar&\cdots\rar&\Hom_\A(M^1,X)\rar&\Hom_\A(L,X)\rar&{}\\
      {}\rar&\Ext_\A^d(N,X)\rar&\Ext_\A^d(M^d,X).
    \end{tikzcd}
  \]
\end{proposition}
\begin{proof}
  The proof of \cite[Lemma 3.5]{iyama_cluster_2011} carries over.
\end{proof}

In view of \th\ref{not_long}, it is natural to consider the following
class of $d$-cluster-tilting subcategories which are better behaved
from the viewpoint of higher homological algebra.

\begin{defprop}
  \th\label{d-ct_h} Let $\A$ be a dualizing $\kk$-variety and
  $\M\subseteq\mod\A$ a $d$-cluster-tilting subcategory. Then, we say
  that $\M$ is \emph{$d\ZZ$-cluster-tilting} if it satisfies the following
  equivalent conditions.
  \begin{enumerate}
  \item\label{it:strongly_d-rigid} $\Ext_\A^k(\M,\M)\neq0$ implies
    that $k\in d\ZZ$.
  \item\label{it:d-syzygies} $\Omega^d(\M)\subset\M$.
  \item\label{it:d-cosyzygies} $\Omega^{-d}(\M)\subset\M$.
  \item\label{it:long_exact_sequences} For each $X\in\M$ and for each
    \[
      \begin{tikzcd}[column sep=small]
        0\rar&L\rar&M^1\rar&\cdots\rar&M^d\rar&N\rar&0
      \end{tikzcd}
    \]
    exact sequence in $\mod\A$ whose terms lie in $\M$ there is an exact
    sequence
    \[
      \begin{tikzcd}[column sep=tiny, row sep=tiny]
        0\rar&\Hom_\A(X,L)\rar&\Hom_\A(X,M^1)\rar&\cdots\rar&\Hom_\A(X,M^d)\rar&\Hom_\A(X,N)\rar&{}\\        {}\rar&\Ext_\A^d(X,L)\rar&\Ext_\A^d(X,M^1)\rar&\cdots\rar&\Ext_\A^d(X,M^d)\rar&\Ext_\A^d(X,N)\rar&{}\\
        {}\rar&\Ext_\A^{2d}(X,L)\rar&\Ext_\A^{2d}(X,M^1)\rar&\cdots\rar&\Ext_\A^{2d}(X,M^d)\rar&\Ext_\A^{2d}(X,N)\rar&\cdots.
      \end{tikzcd}
    \]
  \item\label{it:long_exact_sequences_2} For each $X\in\M$ and for each
    \[
      \begin{tikzcd}[column sep=small]
        0\rar&L\rar&M^1\rar&\cdots\rar&M^d\rar&N\rar&0
      \end{tikzcd}
    \]
    exact sequence in $\mod\A$ whose terms lie in $\M$ there is an exact
    sequence
    \[
      \begin{tikzcd}[column sep=tiny, row sep=tiny]
        0\rar&\Hom_\A(N,X)\rar&\Hom_\A(M^d,X)\rar&\cdots\rar&\Hom_\A(M^1,X)\rar&\Hom_\A(L,X)\rar&{}\\        {}\rar&\Ext_\A^d(N,X)\rar&\Ext_\A^d(M^d,X)\rar&\cdots\rar&\Ext_\A^d(M^1,X)\rar&\Ext_\A^d(L,X)\rar&{}\\
        {}\rar&\Ext_\A^{2d}(N,X)\rar&\Ext_\A^{2d}(M^d,X)\rar&\cdots\rar&\Ext_\A^{2d}(M^1,X)\rar&\Ext_\A^{2d}(L,X)\rar&\cdots.
      \end{tikzcd}
    \]
  \end{enumerate}
\end{defprop}
\begin{proof}
  First we show the equivalence between conditions
  \eqref{it:strongly_d-rigid}, \eqref{it:d-syzygies} and
  \eqref{it:d-cosyzygies}. Let $k\in\set{1,\dots,d-1}$ and
  $m\geq0$. Then, there is a sequence of isomorphisms
  \[
    \Ext_\A^k(\Omega^{dm}\M,\M)\cong\Ext_\A^k(\M,\Omega^{-dm}\M)\cong\Ext_\A^{dm+k}(\M,\M).
  \]
  Since the obstructions for the required sequences to be exact are
  precisely extension groups in degrees which are not multiples of $d$
  between $\A$-modules in $\M$, the claim follows immediately from the
  equalities $\M=\lperp{d-1}\M=\M^{\perp_{d-1}}$ and the assumption
  that $\Omega^d(M)\subset\M$ and $\Omega^{-d}(M)\subset\M$. The
  equivalence between \eqref{it:long_exact_sequences}
  (resp. \eqref{it:long_exact_sequences_2}) and conditions
  \eqref{it:strongly_d-rigid}--\eqref{it:d-cosyzygies} follows from
  \th\ref{not_long} and the existence of isomorphisms
  $\Ext_\A^{dm}(X,-)\cong\Ext_\A^{d}(\Omega^{d(m-1)}X,-)$ for all
  $m\geq1$. We leave the details to the reader.
\end{proof}

\begin{remark}
  We note that the equivalent conditions in \th\ref{d-ct_h} appear in
  the construction of Gei\ss, Keller and Oppermann of
  $(d+2)$-angulated categories from $d$-cluster-tilting subcategories
  of triangulated categories, see \cite{geiss_n-angulated_2013}.
\end{remark}

\subsection{$d$-abelian categories}

The class of $d$-abelian categories is meant to abstract the intrinsic
properties of $d$-cluster-tilting subcategories. Before giving the
definition we recall the definition of $d$-exact sequence in an
additive category, a higher analog of the classical notion of short
exact sequence.

\begin{definition}
  \cite[Defs. 2.2 and 2.4]{jasso_n-abelian_2014} Let $\A$ be an
  additive category. A sequence of morphisms in $\A$
  \[
    \begin{tikzcd}[column sep=small]
      0\rar&a_{d+1}\rar&a_d\rar&\cdots\rar&a_1\rar&a_0
    \end{tikzcd}
  \]
  is called \emph{left $d$-exact}\footnote{We borrow this terminology
    from \cite{lin_right_2014}.} if the induced sequence of functors
  \[
    \begin{tikzcd}[column sep=small]
      0\rar&\A(-,a_{d+1})\rar&\A(-,a_d)\rar&\cdots\rar&\A(-,a_1)\rar&\A(-,a_0)
    \end{tikzcd}
  \]
  is exact. We define \emph{right $d$-exact sequences} dually. A
  sequence is \emph{$d$-exact} if it is both left $d$-exact and right
  $d$-exact.
\end{definition}

\begin{remark}
  Note that $1$-exact sequences are nothing but short exact sequences
  in the usual sense.
\end{remark}

\begin{definition}
  \cite[Def. 3.1]{jasso_n-abelian_2014} Let $\A$ be an additive
  category with split idempotents. We say that $\A$ is
  \emph{$d$-abelian} if the following properties are satisfied.
  \begin{enumerate}
  \item[(A1)] For every morphism $f\colon a_1\to a_0$ in $\A$ there
    exists a left $d$-exact sequence
    \[
      \begin{tikzcd}[column sep=small]
        0\rar&a_{n+1}\rar&a_n\rar&\cdots\rar&a_1\rar{f}&a_0.
      \end{tikzcd}
    \]
  \item[(A1)$^\op$] For every morphism $f\colon a_{d+1}\to a_d$ in
    $\A$ there exists a right $d$-exact sequence
    \[
      \begin{tikzcd}[column sep=small]
        a_{d+1}\rar{f}&a_d\rar&\cdots\rar&a_1\rar&a_0\rar&0.
      \end{tikzcd}
    \]
  \item[(A2)] For every epimorphism $f\colon a_1\to a_0$ in $\A$ there
    exists a $d$-exact sequence
    \[
      \begin{tikzcd}[column sep=small]
        0\rar&a_{n+1}\rar&a_n\rar&\cdots\rar&a_1\rar{f}&a_0\rar&0.
      \end{tikzcd}
    \]
  \item[(A2)$^\op$] For every monomorphism $f\colon a_{d+1}\to a_d$ in
    $\A$ there exists a $d$-exact sequence
    \[
      \begin{tikzcd}[column sep=small]
        0\rar&a_{d+1}\rar{f}&a_d\rar&\cdots\rar&a_1\rar&a_0\rar&0.
      \end{tikzcd}
    \]
  \end{enumerate}
\end{definition}

Axioms (A1) and (A1)$^\op$ immediately imply the following statement.

\begin{proposition}
  \th\label{gldim} Let $\A$ be a $d$-abelian category. Then,
  $\gldim\A\leq d+1$ and $\gldim\A^\op\leq d+1$.
\end{proposition}

The following result gives a connection between $d$-cluster-tilting
subcategories and $d$-abelian categories. In fact, it is the main
motivation for the investigation of $d$-abelian categories.

\begin{theorem}
  \th\label{d-ct--d-abelian} \cite[Thm. 3.16]{jasso_n-abelian_2014} Let
  $\A$ be a dualizing $\kk$-variety and $\M\subseteq\mod\A$ a
  $d$-cluster-tilting subcategory. Then, $\M$ is a $d$-abelian
  category.
\end{theorem}

We need the following definition in the statement of
\th\ref{auslander_correspondence_h}.

\begin{definition}
  Let $\A$ be a $d$-abelian category. As usual, we say that an object
  $a\in\A$ is \emph{projective} if for every epimorphism
  $f\colon b\to c$ the induced morphism $\A(a,b)\to\A(a,c)$ is
  surjective.  We say that \emph{$\A$ has $d$-syzygies} if for every
  $a\in\A$ there exist a $d$-exact sequence
  \[
    \begin{tikzcd}[column sep=small]
      0\rar&b \rar&p_{d-1}\rar&\cdots\rar&p_0\rar&a\rar&0
    \end{tikzcd}
  \]
  where $p_0,p_1,\dots,p_{d-1}$ are projective objects in $\A$. With
  some abuse of notation, we sometimes denote $b$ by $\Omega^d a$. The
  notion of \emph{$\A$ having $d$-cosyzygies} and $\Omega^{-d}a$ are
  defined dually.
\end{definition}

We recall the following result.

\begin{theorem}
  \cite[Thm. 5.16]{jasso_n-abelian_2014}
  \th\label{d-ct_h-d-_abelian_h}.  Let $\A$ be an abelian category with
  enough projectives and $\M\subseteq\A$ a $d$-cluster-tilting
  subcategory. If $\Omega^d(\M)\subseteq\M$, then $\M$ is a
  $d$-abelian category with $d$-syzygies.
\end{theorem}

\section{Auslander correspondence}
\label{sec:auslander_correspondence}

In this section we give a proof of
\th\ref{auslander_correspondence}. For readability purposes we divide
the proof in two parts. Note that the implication
\eqref{it:d-ct}$\Rightarrow$\eqref{it:d-abelian} is shown in
\th\ref{d-ct--d-abelian} since $d$-cluster-tilting subcategories are
functorially finite, hence dualizing $\kk$-varieties by
\th\ref{dkv_ff}.

\subsection{Proof of
  \eqref{it:d-abelian}$\Rightarrow$\eqref{it:new-condition} in
  \th\ref{auslander_correspondence}}

In this subsection, we fix a $d$-abelian dualizing $\kk$-variety
$\A$. By \th\ref{gldim}, the inequality $\gldim\A\leq d+1$ holds.
We begin with the following general lemma.

\begin{lemma}
  \th\label{epi_perp} Let $M\in\mod\A$ and
  \[
    \begin{tikzcd}[column sep=small]
      P_{a_1}\rar{P_f}&P_{a_0}\to M \to 0
    \end{tikzcd}
  \]
  a projective presentation of $M$. Then, $M\in\lperp{}\A$ if and only
  if $f$ is an epimorphism in $\A$.
\end{lemma}
\begin{proof}
  Let $g\colon a_0\to a$ be a morphism in $\A$. Then, $fg=0$ if and
  only if $P_fP_g=0$ if and only if there is a commutative diagram
  \[
    \begin{tikzcd}[row sep=small, column sep=small]
      P_{a_1}\rar{P_f}&P_{a_0}\ar{rr}{P_g}\drar[two heads]&&P_{a}\\
      &&M\urar[dotted,swap]
    \end{tikzcd}
  \]
  From this diagram, and since the Yoneda embedding is faithful, it is
  clear that $M\in\lperp{}\A$ if and only if $f$ is an epimorphism.
\end{proof}

% The following lemma is in fact the last time we use the $d$-abelian
% property of $\A$ in the proof.

\begin{proof}[Proof of
  \eqref{it:d-abelian}$\Rightarrow$\eqref{it:new-condition} in
  \th\ref{auslander_correspondence}]
  Let $M\in\lperp{}\A$ and
  \[
    \begin{tikzcd}[column sep=small]
      P_{a_1}\rar{P_f}&P_{a_0}\rar&M\rar&0
    \end{tikzcd}
  \]
  a projective presentation of $M$. By \th\ref{epi_perp} the morphism
  $f$ is an epimorphism. Since $\A$ is a $d$-abelian category, there
  exists a $d$-exact sequence
  \[
    \begin{tikzcd}[column sep=small]
      0\rar&a_{d+1}\rar&a_d\rar&\cdots\rar&a_1\rar{f}&a_0\rar&0.
    \end{tikzcd}
  \]
  By Yoneda's lemma, for each $a\in\A$ there is an isomorphism between
  the complex
  \begin{equation}
	\label{eq:ExtMA}
    \begin{tikzcd}[column sep=small]
      0\rar&{}_\A(P_{a_0},P_{a})\rar&
      {}_\A(P_{a_1},P_{a})\rar&\cdots\rar&{}_\A(P_{a_d},P_a)\rar&{}_\A(P_{a_{d+1}},P_{a})
    \end{tikzcd}
  \end{equation}
  and the acyclic complex
  \[
    \begin{tikzcd}[column sep=small]
      0\rar&\A(a_0,a)\rar&\A(a_1,a)\rar&\cdots\rar&\A(a_d,a)\rar&\A(a_{d+1},a).
    \end{tikzcd}
  \]
  Finally, since for each $k\in\set{0,\dots,d}$ the $\kk$-module
  $\Ext_\A^k(M,P_{a})$ is isomorphic to the homology of the complex
  \eqref{eq:ExtMA} at $\Hom_\A(P_{a_k},P_{a})$, we conclude that
  $M\in\lperp{d}\A$. Therefore $\lperp{}\A\subset\lperp{d}\A$. Dually,
  we have ${}^\perp \A_{\A^op}\subset {}^{\perp_d} \A_{\A^op}$.
\end{proof}

\subsection{Proof of
  \eqref{it:new-condition}$\Rightarrow$\eqref{it:d-auslander} in
  \th\ref{auslander_correspondence}}

\begin{proposition}
  The subcategory $\lperp{}\A$ is a Serre subcategory of $\mod\A$. In
  particular, $\lperp{}\A$ is an abelian category.
\end{proposition}
\begin{proof}
  It is clear that $\lperp{}\A$ is closed under extensions and
  quotients. Let $0\to L\to M\to N\to0$ be a short exact sequence in
  $\mod\A$ such that $M\in\lperp{}\A$. Condition \eqref{it:new-condition} implies that
  $N\in\lperp{}\A\subseteq\lperp{d}\A$. Hence for each $a\in\A$ the
  functor $\Hom_\A(-,P_{a})$ induces an exact sequence
  \[
    \begin{tikzcd}[column sep=small]
      0=\Hom_\A(M,P_{a})\rar&\Hom_\A(L,P_{a})\rar&\Ext^1_\A(N,P_{a})=0.
    \end{tikzcd}
  \]
  Therefore $L\in\lperp{}{\A}$, which is what we needed to show.
\end{proof}

The following result is necessary to establish a certain Gorenstein
property of $\A$ in \th\ref{gorenstein_prop}.

\begin{lemma}
  \th\label{ExtDuality} The contravariant functor $\mod\A\to\mod\A^\op$
  given by
  \[
    \begin{tikzcd}
      M\mapsto \Ext_\A^{d+1}(M,-)|_{\A}
    \end{tikzcd}
  \]
  induces a duality between abelian categories
  $\lperp{}{\A_\A}\to\lperp{}{(\A_{\A^\op})}$.
\end{lemma}
\begin{proof}
  This can be shown as in the proof of
  \cite[6.2(1)]{iyama_tau-categories_2005} since
  $\lperp{}\A\subset\lperp{d}{\A}$ by
  condition \eqref{it:new-condition}. We leave the details to the reader.
\end{proof}

We can now establish the following $(d+1)$-Gorenstein property of
$\A$, \cf \cite[Def. 0.1]{iyama_tau-categories_2005}.

\begin{proposition}
  \th\label{gorenstein_prop} Let $S$ be a simple $\A$-module of
  projective dimension $d+1$. Then, the following statements hold.
  \begin{enumerate}
  \item The $\A$-module $S$ belongs to $\lperp{}{\A}$.
  \item The $\A^\op$-module $\Ext_\A^{d+1}(S,-)|_{\A}$ is simple and
    has projective dimension $d+1$.
  \end{enumerate}
\end{proposition}
\begin{proof}
  Since $\gldim\A\leq d+1$, every submodule of $\A$ has projective
  dimension at most $d$. Therefore
  $S\in\lperp{}\A$. \th\ref{ExtDuality} implies that
  $\Ext_\A^{d+1}(S,-)|_{\A}$ is a simple $\A^\op$-module which belongs to ${}^\perp \A_{\A^op}\subset {}^{\perp_d} \A_{\A^\op}$,
and therefore has projective dimension $d+1$.
\end{proof}

We denote by $I^0(\A)$ the additive closure of the full subcategory of
$\mod\A$ given by the injective $\A$-modules $I$ such that there exist
$a\in\A$ and an injective envelope $P_a\to I$.

\begin{lemma}
  \th\label{I0Ik} Let $a\in\A$ and
  \[
    \begin{tikzcd}[column sep=small]
      (P_a\to
      I^\bul)=(0\rar&P_a\rar&I^0\rar&I^1\rar&\cdots\rar&I^d\rar&I^{d+1}\rar&0)
    \end{tikzcd}
  \]
  a minimal injective coresolution of $P_a$. Then, the following
  statements hold.
  \begin{enumerate}
  \item\label{it:I0Ik} For all $k\in\set{0,1,\dots,d}$ the $\A$-module
    $I^k$ belongs to $I^0(\A)$.
  \item\label{it:IkId1} The equality
    $I^0(\A)\cap(\add I^{d+1})=0$ is satisfied.
  \end{enumerate}
\end{lemma}
\begin{proof}
  \eqref{it:I0Ik} Let $k\in\set{0,1,\dots,d}$ and $I$ an
  indecomposable direct summand of $I^k$. Let $S$ be the socle of
  $I$. Since $S$ is a simple $\A$-module and $P_a\to I^\bul$ is a
  minimal injective coresolution, there is an isomorphism
  $\Ext_\A^k(S,P_a)\cong\Hom_\A(S,I^k)\neq0$. Then, condition \eqref{it:new-condition}
  implies that $S\notin\lperp{}\A$. Hence, there exists $a'\in\A$ such
  that $\Hom_\A(S,P_{a'})\neq0$. Therefore $I$ is a direct summand of
  the injective envelope of $P_{a'}$. This shows that
  $I\in\add I^0(\A)$.
  
  \eqref{it:IkId1} %We can assume that $I^{d+1}\neq0$ since the claim
  % is obvious otherwise.
  It is enough to show that for each simple submodule $S$ of $I^{d+1}$
  there is an equality $\Hom_\A(S,I^0(\A))=0$. Since
  $\Ext_\A^{d+1}(S,P_a)\neq0$, the $\A$-module $S$ has projective
  dimension $d+1$. Given that every submodule of a projective
  $\A$-module has projective dimension at most $d$, we deduce that
  $S\in\lperp{}\A$. Therefore $\Hom_\A(S,\A)=\Hom_\A(S,I^0(\A))=0$ as
  required.
\end{proof}

We need one more technical lemma.

\begin{lemma}
  \th\label{d1d1-simple-condition} Let $a\in\A$ and
  \[
    \begin{tikzcd}[column sep=small]
      (P_a\to
      I^\bul)=(0\rar&P_a\rar&I^0\rar&I^1\rar&\cdots\rar&I^d\rar&I^{d+1}\rar&0)
    \end{tikzcd}
  \]
  a minimal injective coresolution of $P_a$. Then, for each
  $k\in\set{0,1,\dots,d}$ the inequality $\pdim I^k\leq d$ is satisfied.
\end{lemma}
\begin{proof}
  The proof of \cite[Prop. 6.3(1)]{iyama_tau-categories_2005} carries
  over. We reproduce it here for the convenience of the reader. Recall
  that $\gldim\A\leq d+1$, see \th\ref{gldim}. Let
  $k\in\set{0,\dots,d}$ and suppose that $\pdim I^k=d+1$. Since
  minimal projective resolutions in $\mod\A$ exist, see
  \th\ref{minimal_proj_res}, it readily follows that there exist a
  simple $\A$-module $S$ such that $\Tor_{d+1}^\A(S,I^k)\neq0$. Note
  that this implies that $S$ has projective dimension
  $d+1$. \th\ref{tor_homext} then implies that
  \[
    \Hom_{\A^\op}(\Ext_\A^{d+1}(S,-)|_{\A},I^k)\neq0.
  \]
  Given that $\Ext_\A^{d+1}(S,-)|_\A$ is a simple $\A^\op$-module of
  projective dimension $d+1$, see \th\ref{gorenstein_prop}, and
  $P_a\to I^\bul$ is a minimal injective coresolution, there is an
  isomorphism
  \[
    \Ext_{\A^\op}^k(\Ext_\A^{d+1}(S,-)|_{\A},P_a)\cong\Hom_{\A^\op}(\Ext_\A^{d+1}(S,-)|_{\A},I^k)\neq0.
  \]
  This contradicts the fact that $\Ext_\A^{d+1}(S,-)|_\A$ belongs to
  $\lperp{}(\A_{\A^\op})\subset \lperp{d}(\A_{\A^\op})$, see
  \th\ref{gorenstein_prop}. Therefore $I^k$ has projective
  dimension at most $d$.
\end{proof}

As we shall see, the proof of the implication
\eqref{it:new-condition}$\Rightarrow$\eqref{it:d-auslander} in
\th\ref{auslander_correspondence} follows from the following result.

\begin{proposition}
  \th\label{OmegaPd} Let $\A$ be a dualizing $\kk$-variety of global
  dimension $d+1$ such that
  $\lperp{}\A\subset\lperp{d}\A$. Then,
  \[
    \Omega(\mod\A)=\setP{M\in\mod\A}{\pdim M\leq d}.
  \]
\end{proposition}
\begin{proof}
  Since $\gldim\A\leq d+1$, see \th\ref{gldim}, it is clear that
  \[
    \Omega(\mod\A)\subseteq\setP{M\in\mod\A}{\pdim M\leq d}.
  \]
  We now show that the opposite inclusion holds.
  
  Let $M$ be an $\A$-module of projective dimension at most $d$.  By
  \th\ref{auslander-bridger_sequence} there exist an exact sequence
  \[
    \begin{tikzcd}[column sep=small]
      0\rar&\Ext_{\A^\op}^1(\Tr
      M,-)|_{\A}\rar&M\rar&M^{**}\rar&\Ext_{\A^\op}^2(\Tr
      M,-)|_{\A}\rar&0
    \end{tikzcd}
  \]
  and $M^{**}\in\Omega^2(\mod\A)$. Therefore it is enough to show that
  $\Tr M\in\lperp{1}\A$ since $\Omega(\mod A)$ is closed under
  submodules in $\mod A$. Using backwards induction on $k$, we show
  that $\Tr_k M\in\lperp{k}\A$ for $k\in\set{1,\dots,d}$.

  Let
  \[
    \begin{tikzcd}[column sep=small]
      0\rar&P_d\rar&\cdots\rar&P_1\rar&P_0\rar&M\rar&0
    \end{tikzcd}
  \]
  be a projective resolution of $M$. Then we have an exact sequence
  $P_{d-1}^*\to P_d^*\to \Tr_d M\to0$. Applying the $\A$-duality it
  readily follows that $(\Tr_d M)^*=0$, that is
  $\Tr_d M\in\lperp{}\A\subset\lperp{d}\A$.
  
  We need to show that $\Tr_{k+1}M\in\lperp{k+1}\A$ implies
  $\Tr_k M\in\lperp{k}\A$. By \th\ref{technical_lemma} there exists a
  short exact sequence
  \[
    \begin{tikzcd}[column sep=small]
      0\rar&\Ext_\A^k(M,-)|_{\A}\rar{\varphi}&\Tr_k
      M\rar&\Omega\Tr_{k+1}M\rar&0
    \end{tikzcd}
  \]
  such that $\varphi^*=0$. Put $E_k:=\Ext_\A^k(M,-)|_{\A}$. Applying
  $(-)^*$ to this sequence yields an exact sequence
  \[
    \begin{tikzcd}[column sep=small]
      (\Tr_k
      M)^*\rar{\varphi^*}\rar&(E_k)^*\rar&\Ext_\A^1(\Omega\Tr_{k+1}M,-)|_{\A}\cong\Ext_\A^2(\Tr_{k+1}M,-)|_{\A}.
    \end{tikzcd}
  \]
  Since $\varphi^*=0$ and, by assumption,
  $\Ext_\A^2(\Tr_{k+1},-)|_\A=0$, we conclude that $(E_k)^*=0$, that
  is $E_k\in\lperp{}\A\subset\lperp{d}\A$. Finally, for each
  $j\in\set{1,\dots,k}$ there is an exact sequence
  \[
    \begin{tikzcd}[column sep=small]
      0=\Ext_\A^{j+1}(\Tr_{k+1}M,-)|_{\A}\rar&\Ext_\A^j(\Tr_k
      M,-)|_{\A}\rar&\Ext_\A^j(E_k,-)|_{\A}=0.
    \end{tikzcd}
  \]
  Therefore $\Tr_{k}M\in\lperp{k}\A$. This finishes the proof.
\end{proof}

We can now give the proof of the implication
\eqref{it:new-condition}$\Rightarrow$\eqref{it:d-auslander} in
\th\ref{auslander_correspondence}.

\begin{proof}[Proof of \eqref{it:new-condition}$\Rightarrow$\eqref{it:d-auslander} in \th\ref{auslander_correspondence}]
  By \th\ref{gldim} the inequality $\gldim\A\leq d+1$ holds. Moreover,
  since $\A$ is a $d$-abelian dualizing $\kk$-variety if and only if
  so is $\A^\op$, it is enough to show the one-sided condition
  $\domdim\A\geq d+1$. Let $a\in\A$ and
  \[
    \begin{tikzcd}[column sep=small]
      0\rar&P_a\rar&I^0\rar&I^1\rar&\cdots\rar&I^d\rar&I^{d+1}\rar&0
    \end{tikzcd}
  \]
  a minimal injective coresolution of $P_a$. Then, by
  \th\ref{d1d1-simple-condition,OmegaPd} for each
  $k\in\set{0,1,\dots,d}$ the $\A$-module $I^k$ belongs to
  $\Omega(\mod\A)$. In particular, there exist $b\in\A$ and a
  monomorphism $I^k\to P_b$, which splits since $I^k$ is
  injective. Therefore for all $k\in\set{0,1,\dots,d}$ the $\A$-module
  $I^k$ is projective. This shows that $\domdim\A\geq d+1$, whence
  $\A$ is a $d$-Auslander dualizing $\kk$-variety.
\end{proof}

\subsection{Proof of
  \eqref{it:d-auslander}$\Rightarrow$\eqref{it:d-ct} in
  \th\ref{auslander_correspondence}}

We follow closely the proof
\cite[Thm. 2.6]{iyama_auslander-reiten_2008}.  In this subsection, we
fix a $d$-Auslander dualizing $\kk$-variety $\A$. We denote by
$\Q=\Q_\A$ the full subcategory of $\A_{\A}\subseteq\mod\A$ of
projective-injective $\A$-modules and define
\[
  \B:=\setP{b\in\A}{D(P_b^*)\in\Q_\A}.
\]
Then we have an equivalence $D(-)^*\colon\B\to\Q$. Consider the
functors
\begin{equation*}
  \begin{tikzcd}[column sep=small]
    \FF\colon\mod\A\rar&\mod\B&\text{and}&\GG\colon\mod\B\rar&\mod\A
  \end{tikzcd}
\end{equation*}
defined by $\FF M:=M|_\B$ and $\GG X:=\Hom_\B(P_{-}|_\B,X)$. Note that
$\FF$ is an exact functor and $(\FF,\GG)$ is an adjoint pair.  We shall show that $\B$ is a dualizing
$\kk$-variety (\th\ref{B_is_dkv}) and $\M:=\FF\A$ is a
$d$-cluster-tilting subcategory of $\mod\B$, thus proving the
implication \eqref{it:d-auslander}$\Rightarrow$\eqref{it:d-ct} in
\th\ref{auslander_correspondence}.

\begin{lemma}
  \th\label{B_is_dkv} The category $\B$ is a dualizing $\kk$-variety.
\end{lemma}
\begin{proof}
  By \th\ref{dkv_ff} it is enough to show that $\Q_\A$ is functorially
  finite in $\A_\A$.

  Let $P_a\in\A_\A$. Given that $\domdim\A\geq d+1$ there exist a
  projective-injective $\A$-module $P_q$ and a monomorphism
  $P_a\to P_q$. Moreover, every morphism $P_a\to P_{q'}$ such that
  $P_{q'}$ is projective-injective factors through $P_a\to P_q$. This
  shows that $\Q_\A$ is covariantly finite in $\A_\A$. In order to
  show that $\Q_\A$ is contravariantly finite in $\A_\A$, note that
  $D(\Q_\A)=\Q_{\A^\op}$. Since $\domdim\A^\op\geq d+1$, by what we
  have shown above $D(\Q_{\A})$ is covariantly finite in
  $\A_{\A^\op}$.  Since $D$ is a duality, $\Q_\A$ is contravariantly
  finite in $(D\A)_\A$ and hence in $\A_\A$ by \th\ref{lemma_DA}. This
  shows that $\Q_\A$ is functorially finite in $\A_\A$.
\end{proof}

In the following lemma, note that $D(\B_{\B^{\op}})$ is precisely the
full subcategory of $\mod\B$ of all injective $\B$-modules.

\begin{lemma}
  \th\label{FG} The following statements hold.
  \begin{enumerate}
  \item\label{it:section} There is a natural isomorphism
    $\FF\GG\cong \id_{\mod\B}$.
  \item\label{it:equivalences} The functors $\FF$ and $\GG$ induce
    mutually quasi-inverse equivalences
    \[
      \begin{tikzcd}[column sep=small]
        \Q_\A\rar{\sim}&D(\B_{\B^\op})&\text{and}&
        D(\B_{\B^\op})\rar{\sim}&\Q_\A.
      \end{tikzcd}
    \]
  \end{enumerate}
\end{lemma}
\begin{proof}
  \eqref{it:section} Let $X\in\mod\B$. By Yoneda's lemma, for each
  $b\in\B$ there are functorial isomorphisms
  \[
    \begin{tikzcd}
      (\FF\GG X)(b)=(\GG X)(b)\cong\Hom_\B(P_b|_\B,X)\cong X(b).
    \end{tikzcd}
  \]
  It follows that $\FF\GG\cong \id_{\mod\B}$.
  
  \eqref{it:equivalences} Let $b\in\B$; hence $D(P_b^*)\in\Q$. There
  are functorial isomorphisms
  \[
    \FF D(P_b^*)=D\Hom_\A(P_b,-)|_\B\cong D\A(b,-)|_\B=D\B(b,-).
  \]
  On the other hand, for each $a\in\A$ there are functorial
  isomorphisms
  \[
    (\GG D\B(b,-))(a)\cong
    \Hom_\B(P_a|_\B,D\B(b,-))\cong\Hom_{\B^\op}(\B(b,-),DP_{a}|_\B)\cong
    D P_a(b).
  \]
  Therefore $\GG D\B(b,-)\cong D\A(b,-)\cong D(P_b^*)$. The claim
  follows.
\end{proof}

As a first consequence of \th\ref{FG}, we obtain the following result.

\begin{lemma}
  \th\label{M_is_rigid} The following statements hold.
  \begin{enumerate}
  \item\label{it:fully_faithful} The functor $\FF\colon\A_\A\to\mod\B$
    is fully faithful.
  \item\label{it:Mrigid} The subcategory $\M=\FF\A\subseteq\mod\B$ is
    $d$-rigid, that is $\B\subseteq\lperp{d-1}\B$.
  \end{enumerate} 
\end{lemma}
\begin{proof}
  %We begin by proving \eqref{it:fully_faithful}.
  Let $a\in\A$ and
  \begin{equation}
    \label{eq:cores_P}
    \begin{tikzcd}[column sep=small]
      0\rar&P_a\rar&I^0\rar&I^1\rar&\cdots\rar&I^d
    \end{tikzcd}
  \end{equation}
  be part of a minimal injective coresolution of $P_a$. Since
  $\domdim\A\geq d+1$, the injective $\A$-modules $I^0,I^1,\dots,I^d$
  are also projective. Then, by \th\ref{FG}\eqref{it:equivalences}, we
  have an injective coresolution
  \[
    \begin{tikzcd}[column sep=small]
      0\rar&\FF P_a\rar&\FF I^0\rar&\FF I^1\rar&\cdots\rar&\FF I^d.
    \end{tikzcd}
  \]
  By definition, for each
  $k\in\set{1,\dots,d}$ the homology of the complex
  \begin{equation}
	\label{eq:cores_GFP}
    \begin{tikzcd}[column sep=small]
      0\rar&\GG\FF P_a\rar&\GG\FF I^0\rar&\GG\FF
      I^1\rar&\cdots\rar&\GG\FF I^d
    \end{tikzcd}
  \end{equation}
  at $\GG\FF I^k$ is isomorphic to $\Ext_\B^k(\M,\M)$.  Finally, by
  \th\ref{FG}\eqref{it:equivalences}, the complex \eqref{eq:cores_GFP}
  is isomorphic to the acyclic complex \eqref{eq:cores_P}.  This shows
  that $\GG\FF P_a=P_a$, which means that $\FF|_\A$ is fully faithful,
  and that $\M$ is a $d$-rigid subcategory of $\mod\B$.
\end{proof}

We now give the proof of the implication
\eqref{it:d-auslander}$\Rightarrow$\eqref{it:d-ct} in
\th\ref{auslander_correspondence}.

\begin{proof}[Proof of \eqref{it:d-auslander}$\Rightarrow$\eqref{it:d-ct} in
\th\ref{auslander_correspondence}]
By \th\ref{FG,M_is_rigid,char_tilting} it only remains to show that if
$X$ is a $\B$-module such that $X\in\M^{\perp_d}$, then $X\in\M$ (note
that $\M$ contains $\B_\B$ by construction). Let
\begin{equation}
  \label{eq:cores_X}
  \begin{tikzcd}[column sep=small]
    0\rar&X\rar&I^0\rar&I^1\rar&\cdots\rar&I^d\rar&\cdots
  \end{tikzcd}
\end{equation}
be a minimal injective coresolution of $X$. By assumption, applying
$\GG$ to \eqref{eq:cores_X} yields an exact sequence
\[
  \begin{tikzcd}[column sep=small]
    0\rar&\GG X\rar&\GG I^0\rar&\GG I^1\rar&\cdots\rar&\GG I^d.
  \end{tikzcd}
\]
in which $\GG I^0,\GG I^1,\dots,\GG I^d$ are projective.  Since
$\gldim\A\leq d+1$, the $\A$-module $\GG X$ is projective. Thus,
$X\cong\FF\GG X$ belongs to $\M$, see \th\ref{FG}\eqref{it:section}.
\end{proof}

As a consequence of \th\ref{auslander_correspondence} we obtain a
characterization of injective objects in $d$-abelian dualizing
$\kk$-varieties. It should be compared with
\cite[Thm. 3.12]{jasso_n-abelian_2014}.

\begin{corollary}
  \th\label{injectives} Let $\A$ be a $d$-abelian dualizing
  $\kk$-variety and $q\in\A$. Then, the following statements are
  equivalent.
  \begin{enumerate}
  \item\label{it:injective} The object $q$ is injective in $\A$, that
    is for every monomorphism $f\colon b\to c$ in $\A$ the induced
    morphism $\A(c,q)\to\A(b,q)$ is surjective .
  \item\label{it:A-injective} The $\A$-module $P_q$ is injective.
  \item\label{it:preserves_d-kernels} For every left $d$-exact
    sequence
    \[
      \begin{tikzcd}[column sep=small]
        0\rar&a_{d+1}\rar&a_d\rar&\cdots\rar&a_1\rar&a_0
      \end{tikzcd}
    \]
    the sequence
    \[
      \begin{tikzcd}[column sep=small]
        \A(a_0,q)\rar&\A(a_1,q)\rar&\cdots\rar&\A(a_{d+1},q)\rar&0
      \end{tikzcd}
    \]
    is exact.
  \end{enumerate}
\end{corollary}
\begin{proof}
  \eqref{it:injective}$\Rightarrow$\eqref{it:A-injective} Let $q\in\A$
  be an injective object. Since $\A$ is a $d$-Auslander dualizing
  $\kk$-variety, the injective hull of $P_q$ is of the form
  $P_q\to P_{q'}$; this is induced by a monomorphism $q\to q'$, which
  splits. Thus, $P_q$ is a direct summand of $P_{q'}$ and therefore an
  injective $\A$-module.

  \eqref{it:A-injective}$\Leftrightarrow$\eqref{it:preserves_d-kernels}
  The $\A$-module $P_q$ is injective if and only if for every
  $M\in\mod\A$ there is an equality $\Ext_\A^{>0}(M,P_q)=0$.
  Equivalently, noting that $\gldim\A\leq d+1$, the $\A$-module $P_q$
  is injective if and only if for every exact sequence
  \[
    \begin{tikzcd}[column sep=small]
      0\rar&P_{a_{d+1}}\rar&\cdots\rar&P_{a_1}\rar&P_{a_0}\rar&M\rar&0
    \end{tikzcd}
  \]
  the sequence
  \[
    \begin{tikzcd}[column sep=small]
      0\rar&(M,P_q)\rar&(P_{a_0},P_q)\rar&(P_{a_1},P_q)\rar&\cdots\rar&(P_{a_{d+1}},P_q)\rar&0
    \end{tikzcd}
  \]
  is exact.  The claim follows immediately from Yoneda's
  embedding. 
 
  \eqref{it:preserves_d-kernels}$\Rightarrow$\eqref{it:injective} This is
  clear, since every monomorphism in the $d$-abelian category $\A$ is
  the left-most morphism in a $d$-exact sequence.
\end{proof}

\section{Homological Auslander correspondence}
\label{sec:auslander_correspondence_h}

In this section we give a proof of
\th\ref{auslander_correspondence_h}. We only deal with the first
sequence of equivalences since the second sequence follows by duality.

\begin{proof}[Proof of \th\ref{auslander_correspondence_h}]
  \eqref{it:d-abelian_h}$\Rightarrow$\eqref{it:d-auslander_h} Let $\A$
  be a $d$-abelian dualizing $\kk$-variety with $d$-cosyzygies.  The
  fact that $\A$ is $d$-Auslander dualizing $\kk$-variety follows from
  \th\ref{auslander_correspondence}. Let $a\in\A$. By assumption, there
  exists a $d$-exact sequence
  \[
    \begin{tikzcd}[column sep=small]
      0\rar&a\rar&q^0\rar&\cdots\rar&q^{d-1}\rar&\Omega^{-d}a\rar&0
    \end{tikzcd}
  \]
  where $q^0,\dots q^{d-1}$ are injective objects in $\A$. By
  definition, there is an exact sequence in $\mod\A$ of the form
  \[
    \begin{tikzcd}[column sep=small, row sep=small]
      0\rar&P_a\rar&P_{q^0}\rar&\cdots\rar&P_{q^{d-1}}\ar{rr}\drar[two heads]&&P_{\Omega^{-d}a}\rar&M\rar&0\\
      &&&&&\Omega^{-d}P_a\urar[tail]
    \end{tikzcd}
  \]
  where $P_{q^0},\dots,P_{q^{d-1}}$ are injective $\A$-modules by
  \th\ref{injectives}. Finally, \th\ref{epi_perp} implies
  $M\in\lperp{}{\A_\A}$. Therefore
  $\Omega^{-d}(\A_\A)\subseteq\Omega(\lperp{}\A_\A)$.

  \eqref{it:d-auslander_h}$\Rightarrow$\eqref{it:d-ct_h} Let $\A$ be a
  $d$-Auslander dualizing $\kk$-variety satisfying
  $\Omega^{-d}(\A_\A)\subseteq\Omega(\lperp{}\A_\A)$. We use the
  notation of Section \ref{sec:auslander_correspondence}. Thus, we
  denote by $\Q=\Q_\A$ the full subcategory of
  $\A_{\A}\subseteq\mod\A$ of projective-injective $\A$-modules and
  define
  \[
    \B:=\setP{b\in\A}{D(P_b^*)\in\Q_\A}.
  \]
  Recall that the fully faithful functor
  \begin{equation*}
    \begin{tikzcd}[column sep=small]
      \FF\colon\mod\A\rar&\mod\B
    \end{tikzcd}
  \end{equation*}
  defined by $\FF M:=M|_\B$ induces an equivalence
  $\Q_\A\to D(\B_{\B^\op})$, see \th\ref{FG}. Moreover, by (the proof
  of) \th\ref{auslander_correspondence}, we know that $\M:=\FF\A$ is a
  $d$-cluster-tilting subcategory of $\mod\B$.  It remains to show
  that $\Omega^{-d}(\M)\subset\M$. Let $a\in\A$. By assumption, there
  exists an exact sequence in $\mod\A$ of the form
  \[
    \begin{tikzcd}[column sep=small, row sep=small]
      0\rar&P_a\rar&P_{q^0}\rar&\cdots\rar&P_{q^{d-1}}\ar{rr}\drar[two heads]&&P_{\Omega^{-d}a}\rar&M\rar&0\\
      &&&&&\Omega^{-d}P_a\urar[tail]
    \end{tikzcd}
  \]
  where $P_{q^0},\dots,P_{q^{d-1}}$ are injective $\A$-modules and
  $M\in\lperp{}\A_\A$. Applying the exact functor $\FF$ yields an
  exact sequence
  \[
    \begin{tikzcd}[column sep=small, row sep=small]
      0\rar&\FF P_a\rar&\FF P_{q^0}\rar&\cdots\rar&\FF P_{q^{d-1}}\ar{rr}\drar[two heads]&&\FF P_{\Omega^{-d}a}\rar&\FF M\rar&0\\
      &&&&&\FF \Omega^{-d}P_a\urar[tail]
    \end{tikzcd}
  \]
  where $\FF P_{q^0},\dots,\FF P_{q^{d-1}}$ are injective
  $\B$-modules. Thus, $\FF\Omega^{-d}P_a\cong\Omega^{-d}\FF P_a$. We
  claim that $\FF M=0$, hence
  $\Omega^{-d}\FF P_a\cong\FF\Omega^{-d}P_a\cong\FF
  P_{\Omega^{-d}a}\in\M$.
  Equivalently, we show that $\FF P_{q^{d-1}}\to\FF P_{\Omega^{-d}a}$
  is an epimorphism. We need to show that for every $b\in\B$ each
  morphism $P_b\to P_{\Omega^{-d}a}$ factors through
  $P_{q^{d-1}}\to P_{\Omega^{-d}a}$. The situation can be visualized
  in the commutative diagram
  \[
    \begin{tikzcd}[column sep=small]
      &P_b\dar{\forall}\dlar[dotted,swap]{\exists}\\
      P_{q^{d-1}}\rar&P_{\Omega^{-d}a}\rar&M\rar&0
    \end{tikzcd}
  \]
  where the bottom row is exact.  Applying the $\A$-duality to this
  diagram yields the commutative diagram
  \[
    \begin{tikzcd}[column sep=small]
      &P_b^*\dar[leftarrow]{\forall}\dlar[leftarrow,dotted,swap]{\exists}\\
      P_{q^{d-1}}^*\rar[leftarrow]&P_{\Omega^{-d}a}^*\rar[leftarrow]&M^*=0
    \end{tikzcd}
  \]
  where the bottom row is exact (recall that $M\in\lperp{}\A_\A$).
  Thus, it is enough to show that $P_b^*$ is an injective
  $\A^\op$-module. Indeed, $P_\B^*=D\Q_\A=\Q_{\A^\op}$ consists of
  projective-injective $\A^\op$-modules.

  The implication \eqref{it:d-ct_h}$\Rightarrow$\eqref{it:d-abelian_h}
  follows from \th\ref{d-ct_h-d-_abelian_h} since $d$-cluster-tilting
  subcategories are functorially finite, hence dualizing
  $\kk$-varieties by \th\ref{dkv_ff}. This finishes the proof of the
  theorem.
\end{proof}

\section{Examples}
\label{sec:examples}

In this section we provide a fundamental class of examples of
$d\ZZ$-cluster-tilting subcategories.

Let $\kk$ be a field and $A$ a $d$-representation-finite algebra in
the sense of \cite{iyama_$n$-representation-finite_2011}. Thus, $A$ is
a finite dimensional $\kk$-algebra of global dimension $d$ and there
exist a finite dimensional (right) $A$-module $M$ such that
$\M:=\add\M$ is a $d$-cluster-tilting subcategory of $\mod A$. We let
$\Gamma:=\underline{\End}_A(M)$ be the projectively-stable
endomorphism algebra of $M$.  It is shown in
\cite[Thm. 1.21]{iyama_cluster_2011} that
\[
  \U_d(A):=\add\setP{X[d\ell]\in\operatorname{D^b}(\mod
    A)}{X\in\M\text{ and }\ell\in\ZZ}.
\]
is a $d$-cluster-tilting subcategory of $\operatorname{D^b}(\mod A)$.
It is known that $\U_d(A)$ is a dualizing $\kk$-variety and that
$\mod\U_d(A)$ is a Frobenius abelian category such that there is an
equivalence of triangulated categories
\[
  \begin{tikzcd}
	\underline{\mod}\,\U_d(A)\rar{\sim}&\operatorname{D^b}(\mod\Gamma),
  \end{tikzcd}
\]
see \cite[Props. 2.10 and 2.11]{iyama_mutation_2008} and
\cite[Coros. 3.7 and 4.10]{iyama_stable_2013}. Thus, if $\Gamma$ is a
$(d+1)$-representation finite algebra, then $\U_{d+1}(\Gamma)$ induces
a $(d+1)$-cluster-tilting subcategory of $\mod\U_d(A)$ which is
readily seen to be $d\ZZ$-cluster-tilting. This is the case, for example, if $A$
is a $d$-representation-finite algebra of type $\mathbb{A}$ in the
sense \cite[Sec. 5]{iyama_$n$-representation-finite_2011}.

These examples are used in \cite{darpo_selfinjective_2015} for
constructing selfinjective finite dimensional algebras with
$d$-cluster-tilting subcategories with additive generators by
extending the methods of Riedtmann \cite{riedtmann_algebren_1980}. We
also refer the reader to \cite{jasso_higher_2015} where families of
$d\ZZ$-cluster-tilting subcategories are constructed based
partly on the methods of \cite{darpo_selfinjective_2015}.

%%% Local Variables:
%%% mode: latex
%%% TeX-master: "master"
%%% End:

\bibliographystyle{alpha} \bibliography{zotero}

\end{document}